\documentclass{amsart}

\usepackage{amsmath}
\usepackage{amssymb}
\usepackage{amsfonts}
\usepackage{graphicx}
\usepackage{mlmodern}
\usepackage{enumitem}

\usepackage{todonotes}

\usepackage[colorlinks]{hyperref}
\renewcommand\eqref[1]{(\ref{#1})} %Need with hyperref
\newcommand{\R}{\mathbb R}

\graphicspath{ {images/} }
\title[Grand Lorentz Spaces]{Convolution-type operators in grand Lorentz spaces}

\dedicatory{Dedicated to Professor Lars-Erik Persson on  occasion of his 80th birthday}

\author[E. D. Nursultanov]{Erlan D. Nursultanov}
\address{
	Erlan D. Nursultanov:
	\endgraf
          Department of Mathematics and Informatics
	\endgraf
	Lomonosov Moscow State University, Kazakhstan Branch
	\endgraf
	and
	\endgraf
	Institute of Mathematics and Mathematical Modelling, Almaty 
		\endgraf
	Kazakhstan
	\endgraf
	{\it E-mail address} {\rm er-nurs@yandex.kz}}

\author[H. Rafeiro]{Humberto Rafeiro}
\address{
	Humberto Rafeiro:
	\endgraf
	Department of Mathematical Sciences, College of Sciences
	\endgraf
 United Arab Emirates University
	\endgraf
	Al Ain, PO Box 15551
	\endgraf
	United Arab Emirates
	\endgraf
	{\it E-mail address} {\rm rafeiro@uaeu.ac.ae (corresponding author)}}

\author[D. Suragan]{Durvudkhan Suragan}
\address{
	Durvudkhan Suragan:
	\endgraf
	Department of Mathematics
	\endgraf
 Nazarbayev University
	\endgraf
	53 Kabanbay Batyr Ave, Astana 010000
	\endgraf
	Kazakhstan
	\endgraf
	{\it E-mail address} {\rm durvudkhan.suragan@nu.edu.kz}}

\usepackage{enumitem}

\subjclass[2010]{46E30, 46B70, 26D15}
\keywords{Grand Lorentz space, Hardy-Littlewood-Sobolev inequality, Interpolation theorem, Dual Space}

\thanks{Humberto Rafeiro thanks Nazarbayev University for hosting a two-week research visit in April 2024, made possible by the support of the AUA Scholars Award Program 2023–-2024.}

\newtheoremstyle{theorem}%name
{10pt}          % space above
{10pt}  % space below
{\sl}  % bofy font
{\parindent}     % ident - empty=no indent,  \parindent= paragraph indent
{\bf}  % thm head font
{. }    % punctuation after thm head
{ }    % space after thm head: `` ``=normal \newline=linebreak
{}     % thm head specification
\theoremstyle{theorem}

\numberwithin{equation}{section}
\theoremstyle{plain}
\newtheorem{thm}{Theorem}[section]

\newtheorem{cor}[thm]{Corollary}
\newtheorem{lem}[thm]{Lemma}

\theoremstyle{definition}

\newtheorem{rem}[thm]{Remark}
\newtheorem{ex}[thm]{Example}

\newtheoremstyle{defi}%name
{10pt}          % space above
{10pt}  % space below
{\rm}  % bofy font
{\parindent}     % ident - empty=no indent,  \parindent= paragraph indent
{\bf}  % thm head font
{. }    % punctuation after thm head
{ }    % space after thm head: `` ``=normal \newline=linebreak
{}     % thm head specification
\theoremstyle{defi}
\newtheorem{definition}[thm]{Definition}
\newtheorem{remark}[thm]{Remark}

\def\proofname{\indent {\sl Proof.}}

\begin{document}

		\begin{abstract}

 We introduce and study a novel grand Lorentz space---that we believe is appropriate for critical cases---that lies ``between'' the Lorentz--Karamata space and the recently defined grand Lorentz space from \cite{AFH2020}. %We believe that these spaces are appropriate for critical cases.
 %work well in critical cases.
	We prove both Young’s and O'Neil’s inequalities in the newly introduced grand Lorentz spaces, which allows us to derive a Hardy--Littlewood--Sobolev-type inequality. We also discuss K\"othe duality for grand Lorentz spaces, from which we obtain a new K\"othe dual space theorem in grand Lebesgue spaces.   
	\end{abstract}
	\maketitle
	
\section{Introduction}
In the last two decades, the theory of grand Lebesgue spaces, denoted by $L^{p),\theta}(\Omega)$ and defined as the set of all $f \in L^0(\Omega)$---the space of all real-valued measurable functions on $\Omega$---such that  
\begin{equation}\label{eq:IwaniecSbordone}
    \|f\|_{L^{p),\theta}(\Omega)} := \sup_{0< \varepsilon <p-1}  \varepsilon^{\theta} \|f\|_{L^{p-\varepsilon}(\Omega)} < \infty, \quad |\Omega| < \infty,
\end{equation}is one of the intensively developing directions in modern analysis. The notion of grand Lebesgue spaces $L^{p)} (\Omega):= L^{p),1}(\Omega)$ was introduced in 1992 by Iwaniec and Sbordone \cite{Iwaniec1992}  to deal with the problem of the integrability of the Jacobian under minimal hypothesis. 
An alternative characterisation, given in \cite{Fiorenza2004,Fiorenza2018},  for grand Lebesgue spaces 
 $L^{p),\theta}(\Omega)$ with $|\Omega|=1$ is also known, viz.  
 \begin{equation}\label{eq:grandlebesgueviarearrangement}
\|f\|_{L^{p),\theta} (\Omega)}
\approx \sup_{0<s<1}(1-\ln s)^{-\frac{\theta}{p}}\left(\int_s^1\left(f^*(t)\right)^p \frac{dt}t\right)^{\frac{1}{p}}.
 \end{equation} 
The notion of small Lebesgue spaces $L^{(p,\theta}(\Omega)$, which are a useful companion to $L^{p),\theta}(\Omega)$ as shown in \cite{Fiorenza2000}, can be characterised in terms similar to \eqref{eq:grandlebesgueviarearrangement}, namely
\begin{equation*}
\|f\|_{L^{(p',\theta} (\Omega)}
\approx \int_0^1 (1-\ln s)^{-\frac{\theta}{p'} + \theta-1}\left(\int_0^s\left(f^*(t)\right)^{p'} \frac{dt}t\right)^{\frac{1}{p'}}\frac{ds}{s},    
\end{equation*}
as done in \cite{Fiorenza2018}.

  The spaces $L^{p),\theta}(\Omega)$ have been used in some problems related to PDEs; see, e.g., \cite{Iwaniec1994, Iwaniec1998, Sbordone1996, Sbordone1998}, among others. It is worth pointing out that grand spaces are the appropriate ambient space in which some non-linear equations have to be considered; see \cite{Fiorenza1998,Greco1997}. The theory of operators in grand Lebesgue spaces has been extensively studied in recent years; many of these results are compiled in the monographs \cite{KMRS2016, KMRS2024} dedicated to this subject. 
  The idea of aggrandisation can be applied to other function spaces. For example, the notion of grand Morrey spaces, with only one aggrandised parameter, goes back to \cite{Meskhi1,Meskhi2}, while the grand Morrey space with both aggrandised parameters appeared in \cite{rafeiro1}. For further aggrandised function spaces, see \cite{KMRS2016, KMRS2024} and references therein. We also refer to \cite{NurSur} for boundedness results for convolution-type operators in generalized Morrey spaces.

The Lorentz spaces $L_{p,q}$, introduced in the 1950s by G. Lorentz, are well known, and their importance stems, among other reasons,  from being a proper substitute for Lebesgue spaces in some endpoint estimates in interpolation theory. For a classical reference, we refer to \cite{bennett1988interpolation}, while for a more contemporary source, see \cite{Castillo2021}.

In this paper, we introduce the notion of grand Lorentz spaces $GL^\theta_{p,q}(\Omega)$, see \eqref{eq:2.1} and \eqref{eq:2.2}, where the aggrandisation is applied to the principal index $p$, % in Lorentz spaces, 
which differs from the approach used in \cite{AFH2020}. A key property of the newly introduced grand Lorentz spaces is that classical Lorentz spaces are nested in between grand Lorentz spaces with positive and negative $\theta$, see
\eqref{eq:nestingpropertiesogGrandLorentzspace}, which allows for a more refined control in critical cases.

This paper is organised as follows: in Section \ref{sec:2}, we introduce grand Lorentz spaces and derive some of its elementary properties and embeddings. A comparison with the Lorentz--Karamata space and another version of the grand Lorentz space is also provided; see Lemma \ref{lema:2.3}.  Section \ref{sec:3} is devoted to the study of the O'neil inequality with an application to the boundedness of a modified fractional integral operator; see Theorem \ref{Hardy-Littlewood-Sobolev-type}. The Section \ref{sec:4} deals with Young's inequality while Section \ref{sec:5} examines briefly the topic of interpolation in grand Lorentz spaces. In Section \ref{sec:6} we study K\"othe dual spaces, also know as associate spaces, for grand Lorentz and grand Lebesgue spaces.    

\medskip

We denote \( A \lesssim B \) to indicate that \( A \leq C B \) for some constant \( C > 0 \), which remains independent of the primary parameters under consideration. Furthermore, we use \( A \asymp B \) to express that both \( A \lesssim B \) and \( B \lesssim A \) hold true.

	\section{Grand Lorentz spaces} \label{sec:2}

%Let $\Omega \subset \mathbb{R}^n$ be a measurable set: $|\Omega|=\mu (\Omega)<\infty$.

In what follows, by $\Omega$ we denote a finite measurable subset of $\mathbb R^n$, i.e. $\Omega \subset \mathbb{R}^n$ and $|\Omega|<\infty$. This will be tacitly assumed throughout the entire paper, without further mention.

\begin{definition}\label{def:1}
Let $\theta \in \mathbb{R},\; 0<p\leqslant \infty,$  and $0<q \leqslant \infty$.  The \emph{grand Lorentz space} $GL_{p,q}^\theta(\Omega)$ is the set of measurable functions $f$ for which the following quasinorms are finite
\begin{equation}\label{eq:2.1} 
    \|f\|_{G L_{p,q}^\theta(\Omega)}=
\begin{cases}
\displaystyle \sup _{0<\varepsilon < 1} \varepsilon^\theta \left(\int_0^1\left(t^{\frac{1}{p}+\varepsilon} f^*(t|\Omega|)\right)^q \frac{d t}{t}\right)^{\frac{1}{q}},&  \text{for}\; \theta\geqslant 0,  \\
\displaystyle \sup _{0<\varepsilon < 1} \varepsilon^\theta \left(\int_0^1\left(t^{\varepsilon} f^*(t|\Omega|)\right)^q \frac{d t}{t}\right)^{\frac{1}{q}},&  \text{for}\; \theta > 0 \;\text{and} \;  p=\infty, \\
\displaystyle \inf _{0<\varepsilon<\frac{1}{p}} \varepsilon^\theta \left(\int_0^1\left(t^{\frac{1}{p}-\varepsilon} f^*(t|\Omega|)\right)^q \frac{d t}{t}\right)^{\frac{1}{q}}, & \text{for}\; \theta< 0,
\end{cases}
\end{equation}
if $0<q<\infty$, and
\begin{equation}\label{eq:2.2}
 \|f\|_{G L_{p,\infty}^\theta(\Omega)}=\left\{\begin{array}{l}
\displaystyle \sup _{0<\varepsilon < 1}\sup_{ 0<t<1} \varepsilon^\theta t^{\frac{1}{p}+\varepsilon} f^*(t|\Omega|),\; \; \text{for}\; \theta\geqslant 0,  \\
\displaystyle \sup _{0<\varepsilon < 1}\sup_{ 0<t<1} \varepsilon^\theta t^{\varepsilon} f^*(t|\Omega|),\; \; \text{for}\; \theta >  0 \; \text{and} \; p=\infty,  \\
\displaystyle \inf _{0<\varepsilon<\frac{1}{p}}  \sup_{0<t<1} \varepsilon^\theta t^{\frac{1}{p}-\varepsilon} f^*(t|\Omega|),\;
\; \text{for}\; \theta< 0,
\end{array}\right.
\end{equation}
for $q=\infty$.

\end{definition}

\begin{remark}  
The grand Lorentz spaces $G L_{p, q}^\theta(\Omega)$, unlike the classical Lorentz spaces $L_{p, q}(\Omega)$, allow to consider the scale of spaces $G L_{\infty, q}^\theta (\Omega)$ (close to the space $L_{\infty}$). In fact, this property is one of the motivations for studying the spaces $G L_{p, q}^\theta(\Omega)$. From the definition of $GL_{p,q}^\theta(\Omega)$ follows its connection with the classical Lorentz space $L_{p,q}(\Omega)$, %  when $|\Omega|<\infty$. That is, 
i.e. %when $\theta=0$, we have
$L_{p,q}(\Omega)=G L_{p,q}^0(\Omega)$ with $0<p<\infty$ and $0<q\leqslant\infty$.
\end{remark}

For $\theta>0$, the nesting property
\begin{equation}\label{eq:nestingpropertiesogGrandLorentzspace}
GL_{p,q}^{-\theta} (\Omega)\hookrightarrow L_{p,q}(\Omega)  \hookrightarrow GL_{p,q}^\theta(\Omega),\end{equation}
holds,  and from it, by taking $q=p$, we obtain % In particular, %Particularly,
$$GL_{p, p}^{-\theta} (\Omega)\hookrightarrow L_p(\Omega) \hookrightarrow GL_{p, p}^\theta(\Omega).$$

% \todo[inline]{\texttt{I suggest to delete the following text:} \\
% Observe also that
% $$GL_{p, p}^\theta(\Omega) \hookrightarrow L^{p), \theta}(\Omega)$$
% and
% $$GL_{p, p}^{-\theta}(\Omega) \hookleftarrow L^{(p, \theta}(\Omega).$$
% \texttt{The first embedding appears, with the embedding symbol reversed, in  \eqref{eq:embeddinggrandLebesgue}.  }}

In the case  $\textit{q}=\infty,$ there is an alternative characterisation of %the spaces 
$GL_{p,q}^\theta(\Omega)$. 

	\begin{lem}\label{Lemma 1} Let $|\Omega|<\infty$ and $0<p<\infty$. Then,  for $\theta>0$, we have
\begin{equation}\label{1.2}\|f\|_{G L_{p, \infty}^\theta(\Omega)} \asymp \sup _{t>0} \frac{t^{\frac{1}{p}}}{|\ln t|^\theta} f^*(t |\Omega |),
\end{equation}
\begin{equation}\label{1.3}\|f\|_{G L_{p, \tau}^\theta(\Omega)} \lesssim\left( \int_0^1\left( \frac{t^{\frac{1}{p}}}{|\ln t|^\theta} f^*(t |\Omega |)\right)^\tau\frac{dt}t\right)^{\frac1\tau},
\end{equation}
and
\begin{equation}\label{1.4}\|f\|_{G L_{p,\tau}^{-\theta}(\Omega)} \gtrsim  \left(\int_0^1\left( t^{\frac1p}|\ln t|^\theta f^*(t |\Omega |)\right)^\tau\frac{dt}t \right)^{\frac1\tau}.\end{equation}
\end{lem}
\begin{proof} Consider the function
$$\varphi(\varepsilon)=\varepsilon^\theta t^{\frac{1}{p}+\varepsilon \,\operatorname{sign} \theta},\quad 0<t \leqslant 1,  \quad 0 <\varepsilon \leqslant 1.$$
From the equation
$$
\varphi^{\prime}(\varepsilon)=\theta \varepsilon^{\theta-1} t^{\frac{1}{p}+\varepsilon \operatorname{sign} \theta}+\varepsilon^\theta t^{\frac{1}{p}+\varepsilon \operatorname{sign} \theta} \operatorname{sign} \theta \ln t=0,
$$
we get $\varepsilon=\frac{|\theta|}{|\ln t|}$.

Hence, for $\theta>0$
$$\sup _{0<\varepsilon \leq1} \varphi(\varepsilon) \asymp \frac{t^{\frac{1}{p}}}{|\ln t|^\theta}$$
and for $\theta <0$
$$\inf _{0<\varepsilon <\frac1 p} \varphi(\varepsilon) \asymp \frac{t^{\frac{1}{p}}}{|\ln t|^\theta}.$$
Thus
\begin{align*}
\|f\|_{G L_{p, \infty}^\theta(\Omega)}&=\sup _{0<\varepsilon \leqslant 1}  \sup _{t>0}  \varepsilon^\theta t^{\frac{1}{p}+\varepsilon } f^*(t|\Omega|)  
\\  & =\sup _{t>0}\left(\sup _{0< \varepsilon \leqslant 1} \varepsilon^\theta t^{\frac{1}{p}+\varepsilon }\right) f^*(t |\Omega|)\nonumber \\  & \asymp \sup _{t>0} \frac{t^{\frac{1}{p}}}{|\ln t|^\theta} f^*(t |\Omega|).  
\end{align*}
We also have 
\begin{align*}
\|f\|_{G L_{p, \tau}^\theta(\Omega)}&=\sup _{0<\varepsilon \leqslant 1} \varepsilon^\theta \left(\int_0^1\left( t^{\frac{1}{p}+\varepsilon } f^*(t|\Omega|)\right)^\tau\frac{dt}t\right)^{\frac1\tau}  
\\  & \leqslant\left(\int_0^1\left(\left(\sup _{0<\varepsilon \leqslant 1} \varepsilon^\theta t^{\frac{1}{p}+\varepsilon }\right) f^*(t |\Omega|) \right)^\tau\frac{dt}t\right)^{\frac1\tau}  \\  & \asymp \left( \int_0^1\left( \frac{t^{\frac{1}{p}}}{|\ln t|^\theta} f^*(t |\Omega |)\right)^\tau\frac{dt}t\right)^{\frac1\tau},
\end{align*}
and
\begin{align}
\|f\|_{G L_{p, \tau}^{-\theta} (\Omega)}&=\inf _{0<\varepsilon < \frac{1}{p}} \varepsilon^{-\theta} \left(\int_0^1\left( t^{\frac{1}{p}-\varepsilon } f^*(t|\Omega|)\right)^\tau\frac{dt}t\right)^{\frac1\tau}  \nonumber
\\  & \geqslant \left(\int_0^1\left( \left(\inf _{0<\varepsilon < \frac1 p}\varepsilon^{-\theta}\; t^{\frac{1}{p}-\varepsilon }\right) f^*(t|\Omega|)\right)^\tau\frac{dt}t\right)^{\frac1\tau} \nonumber \\  & \asymp  \left(\int_0^1\left( t^{\frac1p}|\ln t|^\theta f^*(t |\Omega |)\right)^\tau\frac{dt}t \right)^{\frac1\tau}, \nonumber
\end{align}
which ends the proof.
\end{proof}

\begin{ex} Let $\theta>0$,  $\delta>0$, and $\Omega\subset \mathbb{R}^{n}$ with  $|\Omega|=1$. If $f: \Omega\rightarrow \mathbb{R}$ is such that 
\[f^*(t)\asymp \frac{|\ln t|^{\theta-\frac1\tau-\delta}}{t^{\frac{1}{p}}},\]
then  $f \in G L_{p,\tau}^{\theta}(\Omega).$
\end{ex}

From the definition of $GL_{p,q}^\theta(\Omega)$ and well-known  properties of Lorentz spaces, the following properties hold:

\medskip

\begin{enumerate}[label=(P.\arabic*), ref=P.\arabic*] % Define custom label
    \item \label{p1}  If $\theta \leqslant \theta_1$, then
$GL_{p,q}^\theta(\Omega)  \hookrightarrow G L_{p, q}^{\theta_1}(\Omega).$

\item \label{p2}  If $q<q_{1}$, then $GL_{p,q}^\theta(\Omega)  \hookrightarrow G L_{p,q_{1}}^{\theta}(\Omega).$ 

\item \label{p3} If $p<p_1, \; q< q_1 \in(0, \infty],$ then $GL_{p_1, q_1}^\theta(\Omega) \hookrightarrow
GL_{p,q}^\theta(\Omega).$  

\item \label{p4} Let $0<\delta<1, \;   \theta>0$. Then
\begin{equation} \label{property4_1}
\|f\|_{G L_{p, q}^\theta} \asymp \sup _{0<\varepsilon \leqslant \delta} \varepsilon^\theta\left(\int_0^1\left(t^{\frac{1}{p}+\varepsilon} f^*(t |\Omega|)\right)^q \frac{d t}{t}\right)^{1 / q},
\end{equation}
and for $0<\delta<\frac1{p}$
\begin{equation} \label{property4_2}
\|f\|_{G L_{p, q}^{-\theta}(\Omega)} \asymp \inf _{0<\varepsilon \leqslant \delta} \varepsilon^{-\theta}\left(\int_0^1\left(t^{\frac{1}{p}-\varepsilon} f^*(t |\Omega|)\right)^q \frac{d t}{t}\right)^{1 / q}. 
\end{equation}

\medskip

\item \label{p5} \label{itemfive}  If $0<p<\infty$,  $0<\tau\leqslant\infty$, and $\theta>0$, then
\begin{equation} \label{property5_2}
\|f\|_{G L_{p, \tau}^{\theta}(\Omega)} \asymp \sup _{1< k}k^{-\theta } \left(\sum_{m=-\infty}^0\left(2^{m(1/p+1/k)}f^*(2^m |\Omega|) \right)^\tau\right)^{\frac1\tau}
\end{equation}
and 
\begin{equation} \label{property5_1}
\|f\|_{G L_{p, \tau}^{-\theta}(\Omega)} \asymp \inf _{p<  k}k^\theta \left(\sum_{m=-\infty}^0\left(2^{m(1/p-1/k)}f^*(2^m |\Omega|) \right)^\tau\right)^{\frac1\tau}.
\end{equation}

\medskip

\item \label{p6} \textbf{H{\"o}lder's inequality} (cf. \cite[Theorem 2.61]{Castillo2021} for an explicit formulation of the constant in generalised Hölder's inequality in Lorentz spaces). \emph{Let $1<p, p_1, p_2<\infty$,  $1 \leqslant q, q_1, q_2 \leqslant \infty$,
$\theta>0$, 
$\frac{1}{p}
=\frac{1}{p_1}+\frac{1}{p_2}$, $ \frac{1}{q}=\frac{1}{q_1}+\frac{1}{q_2}$, and  $\theta=\theta_1+\theta_2 .$
Then
\begin{equation}\label{Lorentz-Hold}
\|f  g\|_{{L}_{p,q}(\Omega)} \lesssim \|f\|_{G L_{p_1, q_1}^{\theta}(\Omega)}\|g\|_{G L_{p_2, q_2}^{-\theta}(\Omega)}.
\end{equation}}
In particular,
\begin{equation} \label{Hold} \int_{\Omega}|f(x) g(x)| d x \leqslant\|f\|_{GL_{p,q}^\theta(\Omega)}\|g\|_{GL^{-\theta}_{p', q'}(\Omega)},
\end{equation} 
where $p', q'$ are conjugate parameters to $p, q$. For a proof of \eqref{Hold}, see the first part of the proof of Theorem \ref{theo6.1}. The same argument can be extended to show \eqref{Lorentz-Hold}, using the generalised H\"older inequality in Lorentz spaces.
\end{enumerate}

\medskip

We now recall the following function spaces, which will be used subsequently. 

%Now we compare $GL_{p,q}^\theta(\Omega)$ with other similar well-known functional spaces. 

\begin{description}
    \item[Lorentz--Karamata space] For $\theta\in \R$,  the Lorentz--Karamata space $L_{p,\tau,\theta} (\Omega)$ consists of  $f\in L^0(\Omega)$ %those real-valued Lebesgue measurable functions $f$ on $\Omega$,
    for which the quasi-norm
\begin{equation}\label{Lorentz-Karamataspaces}
\|f\|_{L_{p,\tau,\theta} (\Omega)}=\left(\int_0^1\left(t^{1/p}(1-\ln t)^{-\theta}f^*(t|\Omega|)\right)^\tau\frac{dt}t\right)^{\frac1\tau}
\end{equation}
is finite (with the usual modification when $\tau=\infty$); see \cite{GOT2005}. 

\item[Grand Lorentz space] For $\theta>0$, the grand Lorentz space $ L^{p),\tau, \theta}(\Omega) $ is
formed by $f \in L^0 (\Omega)$ % those real-valued Lebesgue measurable functions $f$ on $\Omega$, 
for which
the quasi-norm
\begin{equation}\label{GrandLorentzspaces}
\|f\|_{L^{p),\tau,\theta} (\Omega)}
=\sup_{0<s<1}(1-\ln s)^{-\theta}\left(\int_s^1\left(t^{1/p}f^*(t|\Omega|)\right)^\tau\frac{dt}t\right)^{\frac1\tau}
\end{equation}
is finite; see \cite{AFH2020}.

% \item[Small Lorentz space] For $\theta>0$, the small Lorentz space $L^{(p,r,\theta}(\Omega)$
%  consists of all those real-valued
% Lebesgue measurable functions $f$ on $\Omega$, for which the quasi-norm
% \begin{equation}\label{smallLorentzspaces}
% \|f\|_{L^{(p,\tau,\theta} (\Omega)}
% =\int_0^1(1-\ln s)^{\theta}\left(\int_0^s\left(t^{1/p}f^*(t|\Omega|)\right)^\tau\frac{dt}t\right)^{\frac1\tau}\frac{ds}s
% \end{equation}
% is finite; see \cite{AFH2020}. 
\end{description}
% \todo[inline]{I suggest to take the definition of small Lorentz spaces given above, since we are not using it.}

In the next lemma, we obtain some relations between the newly introduced grand Lorentz spaces in \eqref{eq:2.1}-\eqref{eq:2.2} and the spaces defined in \eqref{Lorentz-Karamataspaces} and  \eqref{GrandLorentzspaces}.

\begin{lem}\label{lema:2.3} 
Let %$|\Omega|<\infty,$ 
$\theta > 0$,  $0<p<\infty$, and   $0<\tau\leqslant\infty$.

If $\tau=\infty $ , then
\begin{equation} \label{1.8}
GL_{p,\infty}^\theta(\Omega)=L^{p),\infty,\theta} (\Omega)=L_{p,\infty,\theta} (\Omega).
\end{equation}
 If $0< \tau<\infty $ , then
\begin{equation} \label{1.9}
L_{p,\tau,\theta} (\Omega)\hookrightarrow GL_{p,\tau}^\theta(\Omega)\hookrightarrow L^{p),\tau,\theta} (\Omega)
\end{equation}
and
\begin{equation} \label{1.10}
GL_{p,\tau}^{-\theta}(\Omega)\hookrightarrow L_{p,\tau,-\theta} (\Omega).
\end{equation}
\end{lem}

\begin{proof} Relation \eqref{1.8} follows from the definition of the spaces and \eqref{1.2}.

We now verify 
\eqref{1.9}. Let  $f\in G L_{p, \tau}^\theta(\Omega)$. From \eqref{property5_2} it follows
\begin{align*}
    \|f\|_{G L_{p, \tau}^\theta(\Omega)}&\asymp \sup _{1< k}k^{-\theta } \left(\sum_{m=-\infty}^0\left(2^{m(1/p+1/k)}f^*(2^m |\Omega|) \right)^\tau\right)^{\frac1\tau}\\
&\geqslant  \sup _{1< k}k^{-\theta } \left(\sum_{m=-k}^0\left(2^{m(1/p+1/k)}f^*(2^m |\Omega|)\right)^\tau\right)^{\frac1\tau}\\
&\geqslant 2^{-1}\sup _{1< k}k^{-\theta } \left(\sum_{m=-k}^0\left(2^{m/p}f^*(2^m |\Omega|) \right)^\tau\right)^{\frac1\tau}\\ & \gtrsim \|f\|_{L^{p),\tau,\theta} (\Omega)}.
\end{align*}
If $f\in L_{p,\tau,\theta} (\Omega) $, then
\begin{align*}
\|f\|_{G L_{p, \tau}^\theta(\Omega)}&=\sup _{0<\varepsilon \leqslant 1} \varepsilon^\theta \left(\int_0^1\left( t^{\frac{1}{p}+\varepsilon } f^*(t|\Omega|)\right)^\tau\frac{dt}t\right)^{\frac1\tau}\\
&\leqslant \left(\int_0^1\left( \sup _{0<\varepsilon \leqslant 1} \varepsilon^\theta t^{\frac{1}{p}+\varepsilon } f^*(t|\Omega|)\right)^\tau\frac{dt}t\right)^{\frac1\tau}\\ 
&\lesssim \|f\|_{L_{p,\tau,\theta} (\Omega)}.%<\infty.
\end{align*}
The embedding \eqref{1.10} follows from analogous considerations as above.
\end{proof}

%The left-hand side inclusion in \eqref{1.9} is strict, as can be seen from the following example.
\begin{ex}Let $|\Omega|=1$ and  $f :\Omega\to\mathbb R $ be such that  $f^*(t)= t^{-\frac1p}|\ln t |^{\theta-\frac1\tau},$ with   $0< \tau, \theta < \infty $. 
 Then
$$
 \int_0^1\left( \frac{t^{\frac{1}{p}}}{|\ln t|^\theta} f^*(t)\right)^\tau\frac{dt}t=\int_0^1\frac{dt}{t\ln t}=\infty,
$$
but
\begin{align*}
\|f\|_{G L_{p, \tau}^\theta(\Omega)}&=\sup _{0<\varepsilon < 1} \varepsilon^\theta \left(\int_0^1\left( t^{\frac{1}{p}+\varepsilon } f^*(t|\Omega|)\right)^\tau\frac{dt}t\right)^{\frac1\tau}\\
&=\sup _{0<\varepsilon < 1} \varepsilon^\theta \left(\int_0^1 \frac{t^{\varepsilon \tau }|\ln t|^{\theta\tau} dt}{t|\ln t|}
\right)^{\frac1\tau}\\
& =\sup _{0<\varepsilon < 1} \varepsilon^\theta \left(\int^{\infty}_0 e^{-\varepsilon s\tau }s^{\theta\tau} \frac{ds}{s}\right)^{\frac1\tau} \\
&= %\sup _{0<\varepsilon < 1} \varepsilon^{\theta (1-\tau)}
\left(\int^{\infty}_0 e^{ -s\tau }s^{\theta\tau} \frac{ds}{s}
\right)^{\frac1\tau}<\infty.
\end{align*}
\end{ex}

The next result shows that the grand Lebesgue space is embedded between appropriate grand Lorentz spaces.
%Note that the grand Lebesgue spaces $L^{p), \theta}(\Omega)$ and the small Lebesgue spaces $L^{(p,\theta}(\Omega)$ were considered in several works; we refer to \cite{Fiorenza2000,Fiorenza2004} and references therein. 
%Thus,  The spaces \( GL_{p,q}^\theta(\Omega) \) are analogues of \( L^{p),\theta}(\Omega) \) for \( \theta > 0 \) and of \( L^{(p,\theta}(\Omega) \) for \( \theta < 0 \).

% \medskip

% \todo[inline]{Please check embedding \eqref{eq:embeddinggrandLebesgue} and the corresponding proof. I could not check the embedding $G L_{p,\,p-\delta}^\theta(\Omega) \hookrightarrow L^{p),\theta}(\Omega)$. The embeddings in the previous proof always used $p-\varepsilon$ but it seems that for grand Lorentz the first index should be $p(\varepsilon)$. Maybe I'm missing something. In any case the embedding given below I think it's correct.}
\begin{lem}  Let %$|\Omega|<\infty,$ 
$0<\theta <\infty$ and $0<p<\infty$. Then
\begin{equation}\label{eq:embeddinggrandLebesgue}
G L_{p+\delta,p}^\theta(\Omega) \hookrightarrow L^{p),\theta}(\Omega) \hookrightarrow G L_{p,\, p}^\theta(\Omega)\end{equation}
for all $\delta>0$. If $0<\delta \leqslant  p-1$, then 
\begin{equation}\label{eq:anotherembedding}
    G L_{p,p-\delta}^\theta(\Omega) \hookrightarrow L^{p),\theta}(\Omega).
\end{equation}
\end{lem}

\begin{proof} Note that $p(\varepsilon) < p-\varepsilon$, for $0 < \varepsilon < p-1$ and $\frac{1}{p(\varepsilon)}:= \frac{1}{p}+ \varepsilon$. The right-hand side embedding in \eqref{eq:embeddinggrandLebesgue} follows from 
\[
L_{p-\varepsilon, p-\varepsilon} (\Omega) \hookrightarrow L_{p(\varepsilon), p } (\Omega).
\]

To prove the other embedding, we first recall the well-known result
\begin{equation}
    \|f\|_{L^{p),\theta}(\Omega)} \asymp \sup _{0< \varepsilon \leqslant \eta}\varepsilon^\theta\|f\|_{L_{p-\varepsilon}(\Omega)},
\end{equation}
where $\eta \in (0, p-1)$; see, e.g.,  argument in the proof of  \cite[Theorem 15.15]{KMRS2016}. Note also that, for $\delta>0$, there exists $\varepsilon_0 >0$ for which the inequality
\[
\frac{1}{(p+\delta)(\varepsilon)}:= \frac{1}{p+\delta} + \varepsilon < \frac{1}{p-\varepsilon}, \quad 0< \varepsilon <\varepsilon_0 
\]
is valid. Taking this into account, the embedding on the left-hand side in \eqref{eq:embeddinggrandLebesgue} now follows from the embedding
\[
L_{(p + \delta)(\varepsilon), p}(\Omega) \hookrightarrow L_{p-\varepsilon, p-\varepsilon}(\Omega).
\]

To prove \eqref{eq:anotherembedding}, note that $p(\lambda \varepsilon) > p-\varepsilon$ whenever $0<\lambda < 1/ p^2$. From the embedding 
\[
L_{p(\lambda \varepsilon), p-\delta }(\Omega) \hookrightarrow L_{p-\varepsilon, p-\varepsilon} (\Omega)
\]
and the elementary equality $\varepsilon^\theta = \lambda^{-\theta} (\lambda \varepsilon)^{\theta}$, the desired result follows.
% \todo[inline]{Old proof: \\
% The right embedding follows from $L_{p-\varepsilon,\, p}(\Omega) \hookrightarrow  L_{p-\varepsilon,\, p}(\Omega)$. To show the left embedding,  we set $|\Omega|=1$ without loss of generality.
% Let $\delta>0$. As was shown in \eqref{property5}:
% $$\|f\|_{L^{p),\theta}(\Omega)} \asymp \sup _{0<\varepsilon \leqslant \delta}\varepsilon^\theta\|f\|_{L_{p-\varepsilon}(\Omega)}.$$
% Thus, the left embedding follows
% from $L_{p-\varepsilon }(\Omega) \hookrightarrow  L_{p-\varepsilon, p-\delta}(\Omega).$}
\end{proof}

\begin{rem}
    Using \eqref{eq:embeddinggrandLebesgue} and the properties \eqref{p1}--\eqref{p3}, we can obtain further embedding relations; we leave the details to the reader.
\end{rem}

\section{O'Neil's inequality in grand Lorentz spaces}\label{sec:3}

For measurable functions \( f \) and \( g \) on \( \mathbb{R}^n \), recall that their convolution is formally defined as
$$
f*g(y)=\int_{\mathbb{R}^n} f(x)g(y-x)dx.
$$
We have the following classical inequalities:
\begin{description}
    \item[Young's inequality]
Let $1\leqslant p,q,r\leqslant\infty$ with $1+\frac1q=\frac1p+\frac1r$. Then
\begin{equation} \label{2.1}
\|f*g\|_{L_{q}(\mathbb{R}^n)}\leqslant \|f\|_{L_{p}(\mathbb{R}^n)}\|g\|_{L_{r}(\mathbb{R}^n)}.
\end{equation}

\medskip

\item[O'Neil's inequality]
Let $1< p,q,r<\infty$, $0<s,s_1,s_2\leqslant \infty$, $1+\frac1q=\frac1p+\frac1r$, and $\frac1s=\frac1{s_1}+\frac1{s_2}$. Then
\begin{equation} \label{2.2}
\|f*g\|_{L_{q,s}(\mathbb{R}^n)}\lesssim  d_{p,q,r}\|f\|_{L_{p,s_1}(\mathbb{R}^n)}\|g\|_{L_{r,s_2}(\mathbb{R}^n)},
\end{equation}
where
\begin{equation} \label{2.3}
d_{p,q,r}=%\max\left(\left(\frac1p-\frac1q\right)^{-1}, \left(\frac1r-\frac1q\right)^{-1}\right). 
\max(p',r').
\end{equation}

\end{description}
It is worth mentioning that the importance of O'Neil's inequality lies, among other things, in its role in extending Hardy--Littlewood--Sobolev theorem on the boundedness of the Riesz operator 
 \begin{equation} \label{2.4}
 I_\alpha f(y)=\int_{\mathbb{R}^n}\frac{f(x)}{|x-y|^{n-\alpha}}dx = (f \ast |\cdot|^{\alpha-n})(y)
 \end{equation}
to Lorentz spaces. 

In what follows, we study the convolution over $\Omega$, still denoted by $f \ast g $ and given by
\[
f \ast g(y) = \int_\Omega f(x) g(y-x)dx,
\]
in Lorentz spaces, where we extend the functions $f$ and $g$ by zero, whenever necessary.

% \todo[inline,  backgroundcolor=red, bordercolor=red, textcolor=black]{I think it's easier to define, unless I’m overlooking something, the convolution of two functions over $\Omega$ as simply extending them by zero outside $\Omega$, instead of via periodicity. I think with this convention, all the results remain true.}

% \todo[inline]{\texttt{I think we should delete the following text, from  previous version:}

% In this section we study the convolution operator 
% \begin{equation} \label{2.5}
%  Af(y)=\int_{[0,1]^n}f(x)g(x-y)dx
%  \end{equation}
%  in grand Lorentz spaces.

% In this case, we assume that $f$ and $g$ are periodic functions with the same periods. In particular, we study the boundedness in grand Lorentz spaces of an operator of the form
% \begin{equation} \label{2.6}
%  I_{\alpha,\beta} f(y)=\int_{[0,1]^n}\frac{f(x)}{|x-y|^{n-\alpha}|\ln |x-y||^\beta}dx
%  \end{equation}
% where $0<\alpha<n$ and $\beta\in \mathbb{R}$.}

We start with a proto-O'Neil-type inequality, where the target and source spaces have different natures, namely classical Lorentz spaces and grand Lorentz spaces.

\begin{lem}\label{Lemma6} Let $1< p<q<\infty$, $ \theta \geqslant 0$,
$\frac{1}{r}=1+\frac{1}{q}-\frac{1}{p}$, and $0<\tau \leqslant \infty$.
Then \begin{equation} \label{conv1}
\|f * g\|_{L_{q,\tau}(\Omega)} \lesssim d_{p,q,r} \|f\|_{GL_{p,\tau}^{-\theta}(\Omega)}\|g\|_{GL^\theta_{r,\infty}(\Omega)}
\end{equation}
and
\begin{equation} \label{conv2}
\|f * g\|_{L_{q,\tau}(\Omega)} \lesssim  d_{p,q,r}\|f\|_{GL_{p,\tau}^{\theta}(\Omega)}\|g\|_{GL^{-\theta}_{r,\infty}(\Omega)},
\end{equation}
where the constant $d_{p,q,r}$ is defined in \eqref{2.3}.
\end{lem}

\begin{proof} Let $ 0 < \varepsilon < \min\left(\frac{1}{p}, \frac{1}{r'}\right)$,
$$
\frac1{p(-\varepsilon)}=\frac1p-\varepsilon,\;\;\;\;\;\frac1{r(\varepsilon)}=\frac1r+\varepsilon.
$$
Thus, $1+\frac1q=\frac1{p(-\varepsilon)}+\frac1{r(\varepsilon)}$, where $1<q, p(-\varepsilon), r(\varepsilon)<\infty$. Hence, we apply O'Neil's inequality \eqref{2.2} to get 
\begin{align*}
\|f * g\|_{L_{q,\tau}(\Omega)} &\lesssim \|f\|_{L_{p(-\varepsilon),\tau}(\Omega)}\|g\|_{L_{r(\varepsilon),\infty}(\Omega)}\\
&=\varepsilon^{-\theta}\|f\|_{L_{p(-\varepsilon),\tau}(\Omega)} \varepsilon^{\theta}\|g\|_{L_{r(\varepsilon),\infty}(\Omega)} \\
&\leqslant\varepsilon^{-\theta}\|f\|_{L_{p(-\varepsilon),\tau}(\Omega)} \|g\|_{GL_{r,\infty}^\theta(\Omega)}.
\end{align*}
Since $\varepsilon>0$ is arbitrary,   we obtain 
$$
\|f * g\|_{L_{q,\tau}(\Omega)} \lesssim \|f\|_{GL_{p,\tau}^{-\theta}(\Omega)}\|g\|_{GL^\theta_{r,\infty}(\Omega)}.
$$
The second inequality in the statement is proved similarly. 
\end{proof}

\begin{thm}[O'Neil-type inequality]\label{O'Neil} Let $1< p< q<\infty$, $ 0\leqslant \theta, \theta_0,\theta_1$, 
$\frac{1}{r}=1+\frac{1}{q}-\frac{1}{p}$, and $0<\tau \leqslant\infty$.

If  $ \theta=\theta_1+\theta_0,$ then
 \begin{equation} \label{2.12}
\|f * g\|_{GL_{q,\tau}^{\theta_1}(\Omega)} \lesssim \|f\|_{GL_{p,\tau}^{-\theta_0}(\Omega)}\|g\|_{GL^\theta_{r,\infty}(\Omega)}
\end{equation}
  and \begin{equation} \label{2.13}
\|f * g\|_{GL_{q,\tau}^{-\theta_1}(\Omega)} \lesssim \|f\|_{GL_{p,\tau}^{\theta_0}(\Omega)}\|g\|_{GL^{-\theta}_{r,\infty}(\Omega)}.
\end{equation}

If  $ \theta=\theta_1-\theta_0,$ then
\begin{equation} \label{2.14}
\|f * g\|_{GL_{q,\tau}^{\theta_1}(\Omega)} \lesssim \|f\|_{GL_{p,\tau}^{\theta_0}(\Omega)}\|g\|_{GL^{\theta}_{r,\infty}(\Omega)}
\end{equation}
and 
\begin{equation} \label{2.15}
\|f * g\|_{GL_{q,\tau}^{-\theta_1}(\Omega)} \lesssim \|f\|_{GL_{p,\tau}^{-\theta_0}(\Omega)}\|g\|_{GL^{\theta}_{r,\infty}(\Omega)}.
\end{equation}
\end{thm}

\begin{proof}
To show \eqref{2.12}, 
let $0<\varepsilon< \min \{\frac1p-\frac1{q}, \;\frac1{2r'}\}$ and 
$$
\frac1{q(\varepsilon)}=\frac1q+\varepsilon,\quad 
\frac1{r(\varepsilon)}=\frac1r+\varepsilon.
$$

Thus, $1 < p< q(\varepsilon) < \infty.$

By using Lemma \ref{Lemma6} and the fact that  $\frac{1}{q(\varepsilon)} + 1 = \frac{1}{p} + \frac1{r(\varepsilon)}$, we have 
$$
\| f*g \|_{L_{q(\varepsilon),\tau}(\Omega)} \lesssim \| f \|_{GL_{p, \tau}^{-\theta_0}(\Omega)} \| g \|_{GL^{\theta_0}_{r(\varepsilon),\infty}(\Omega)}.
$$
%Consider 
Noting that $(r(\varepsilon))(\delta)= r(\varepsilon + \delta)$, we get
\begin{align*}
\varepsilon^{\theta_1}\| g \|_{GL^{\theta_0}_{r(\varepsilon),\infty}(\Omega)} &\asymp\varepsilon^{\theta_1}\sup_{0<\delta< \frac1{2r'}}\delta^{\theta_0}\| g \|_{L_{r(\varepsilon+\delta),\infty}(\Omega)}\\
&\leqslant \sup_{0<\delta< \frac1{2r'}}(\varepsilon + \delta)^{\theta_1}(\varepsilon + \delta)^{\theta_0}\| g \|_{L_{r(\varepsilon+\delta),\infty}(\Omega)}\\ & \lesssim  \|g\|_{GL^\theta_{r,\infty}(\Omega)}.
\end{align*}
Then 
$$
\varepsilon^{\theta_1}\| f*g \|_{L_{q(\varepsilon),\tau}(\Omega)} \lesssim \| f \|_{GL_{p, \tau}^{-\theta_0}(\Omega)} \| g \|_{GL^{\theta}_{r,\infty}(\Omega)}.
$$

Now, from the arbitrariness of $\varepsilon>0$ it follows the desired result.

To show \eqref{2.13}, let $0<\varepsilon< \min\{\frac1{q},\frac1{2r}\}$,
$$
\frac1{q(-\varepsilon)}=\frac1q-\varepsilon,\;\;\;\;\;
\frac1{r(-\varepsilon)}=\frac1r-\varepsilon.
$$

Since $\frac{1}{q(-\varepsilon)} + 1 = \frac{1}{p} + \frac1{r(-\varepsilon)}$ , from Lemma  \ref {Lemma6} it follows that
$$
\| f*g \|_{L_{q(-\varepsilon),\tau} (\Omega)}\lesssim \| f \|_{GL_{p, \tau}^{\theta_0}(\Omega)} \| g \|_{GL^{-\theta_0}_{r(-\varepsilon),\infty}(\Omega)},
$$
yielding 
\begin{align*}
\varepsilon^{-\theta_1}\| g \|_{GL^{-\theta_0}_{r(-\varepsilon),\infty}(\Omega)} &\asymp\varepsilon^{-\theta_1}\inf_{0<\delta<\frac1{2r}}\delta^{-\theta_0}\| g \|_{L_{r(-\varepsilon-\delta),\infty}(\Omega)}\\
&\leqslant \varepsilon^{-\theta_1-\theta_0}\| g \|_{L_{r(-2\varepsilon),\infty}(\Omega)}.
\end{align*}

Thus, we compute 
\begin{align*}
\|f * g\|_{GL_{q,\tau}^{-\theta_1}(\Omega)} &\leqslant \varepsilon^{-\theta_1}\|f*g\|_{L_{q(-\varepsilon),\tau}(\Omega)}\\
&\lesssim \| f \|_{GL_{p, \tau}^{\theta_0}(\Omega)}\varepsilon ^{-\theta_1-\theta_0}\| g \|_{L_{r(-2\varepsilon),\infty}(\Omega)}.
\end{align*}
From the arbitrariness of \( \varepsilon > 0 \), the result follows.

To prove \eqref{2.14}, let $0<\varepsilon< \min \{\frac1{q'},\frac1{r'}\}$.
Lemma \ref{Lemma6} implies
$$
\| f*g \|_{L_{q(\varepsilon),\tau}(\Omega)} \lesssim \| f \|_{GL_{p, \tau}^{\theta_0}(\Omega)} \| g \|_{GL^{-\theta_0}_{r(\varepsilon),\infty}(\Omega)},
$$
thus 
\begin{align*}
\varepsilon^{\theta_1}\| f*g \|_{L_{q(\varepsilon),\tau}(\Omega)} &\lesssim \| f \|_{GL_{p, \tau}^{\theta_0}(\Omega)} \varepsilon^{\theta_1}\inf_{0<\delta<\frac1{r'}}\delta^{-\theta_0}\| g \|_{L_{r(\varepsilon-\delta),\infty}(\Omega)}\\
&\leqslant\| f \|_{GL_{p, \tau}^{\theta_0}(\Omega)} \varepsilon^{\theta_1}(\varepsilon/2)^{-\theta_0}\| g \|_{L_{r(\varepsilon/2),\infty}(\Omega)}\\ &\leqslant 2^{\theta_1}\| f \|_{GL_{p, \tau}^{\theta_0}(\Omega)} \|g\|_{GL^\theta_{r,\infty}(\Omega)},
\end{align*}
which gives the desired result. 

Inequality \eqref{2.15} can be proved as inequality \eqref{2.13}.
\end{proof}

For the next theorem, we introduce the power-logarithmic Riesz operator $I_{\alpha,\theta}$ (over $\Omega \subset \mathbb R^n$) given by
\[
I_{\alpha,\theta} f(y)= \int_\Omega \frac{f(x) |\ln |y-x||^\theta}{|y-x|^{n-\alpha}} dx = (f \ast K_{\alpha, \theta}) (y),  
\]
where $K_{\alpha, \theta} = |x|^{\alpha-n} |\ln |x||^\theta  $.

\begin{thm}\label{Hardy-Littlewood-Sobolev-type}
	Let $1<p<q<\infty,$
$\alpha=\frac{1}{p}-\frac{1}{q},$ $\theta=\theta_1-\theta_0,$ and $\theta_0, \theta_1 ,\theta \geq 0$.
Then the operator $I_{\alpha, \theta}: G L_{p,\tau}^{\theta_0}([0,1]) \longrightarrow  
 G L_{q,\tau}^{\theta_1}([0,1])$ is bounded. 
%$$I_{\alpha, \theta}(f)(y)=\int_0^1 \frac{f(x) \left| \ln | y-x|\right|^\theta}{|y-x|^{1-\alpha}} d x$$ is bounded from $G L_{p,\tau}^{\theta_0}([0,1])$ to $G L_{q,\tau}^{\theta_1}([0,1]).$

In particular, if $\theta_0=0$, then
$I_{\alpha, \theta}: L_p ([0,1]) \longrightarrow GL_{q, p}^{\theta}([0,1])$
is bounded.
\end{thm}

	\begin{proof}%[Proof of Theorem \ref{Hardy-Littlewood-Sobolev-type}] 
From Lemma \ref{Lemma 1}, we have
$K_{\alpha, \theta}(x)=|x|^{\alpha-1}  |\ln |x||^\theta  \in G L_{r, \infty}^{\theta}([0,1])$, where $\alpha=1-\frac{1}{r}.$
The result now follows from Theorem \ref{O'Neil}. 
% that
% \begin{equation*}
% I_{\alpha , \theta}(f)=f*g: GL_{p ,\tau}^{\theta_0} ([0,1])\longrightarrow GL_{q ,\tau}^{\theta_1}([0,1]). \qedhere
% \end{equation*}
\end{proof}
\begin{rem} Theorem \ref{Hardy-Littlewood-Sobolev-type} holds in a general case, that is,

$$
I_{\alpha}=I_{\alpha,0}: G L_{p,\tau}^\theta(\Omega) \longrightarrow G L_{q, \tau}^\theta(\Omega),
$$
where $
\alpha=\frac{n}{p}-\frac{n}{q}.
$
\end{rem}

\section{Young's inequality in grand Lorentz spaces}\label{sec:4}

In this section, we are interested in convolution inequalities for grand Lorentz spaces in limiting (or critical) cases. 

\begin{lem}\label{L3.1}
   Let   $1<p,q,r<\infty$, $1+\frac1q=\frac1p+\frac1r$,
    $\theta\in[0,1]$, $0<\eta,\eta_1,\eta_2\leqslant\infty$, and
    \begin{equation} \label{3.1}
    1+\frac1\eta=\frac1{\eta_1}+\frac1{\eta_2}+\theta.
    \end{equation}
Then
\begin{equation} \label{3.2}
\|f*g\|_{L_{q,\eta}}\lesssim  d_{p,q,r}^\theta\|f\|_{L_{p,\eta_1}}\|g\|_{L_{r,\eta_2}}.
\end{equation}
\end{lem}
\begin{proof}
   Inequality \eqref{3.2} follows from Young's inequality  \eqref{2.1}, O'Neil's inequality \eqref{2.2}, and by using multilinear interpolation theorems (Theorems 4.4.1 and 4.4.2 from \cite{Berg}).
\end{proof}

\begin{lem}\label{L3.2} Let $1< p<q<\infty$, $ \theta \geqslant 0, \bar\theta\in[0,1], \;$
$\frac{1}{r}=1+\frac{1}{q}-\frac{1}{p}$ , $\;0<\tau ,\tau_0, \tau_1\leqslant \infty$, and
$1+\frac1{\tau_1}=\frac1{\tau_0}+\frac1{\tau}+\bar\theta$.
Then \begin{equation} \label{3.3}
\|f * g\|_{L_{q,\tau_1}(\Omega)} \lesssim d_{p,q,r}^{\bar\theta} \|f\|_{GL_{p,\tau_0}^{-\theta}(\Omega)}\|g\|_{GL^\theta_{r,\tau}(\Omega)}.
\end{equation}
and
\begin{equation} \label{3.4}
\|f * g\|_{L_{q,\tau_1}(\Omega)} \lesssim d_{p,q,r}^{\bar\theta} \|f\|_{GL_{p,\tau_0}^{\theta}(\Omega)}\|g\|_{GL^{-\theta}_{r,\tau}(\Omega)}.
\end{equation}
where the constant $d_{p,q,r}$ is defined in  \eqref{2.3}.
\end{lem}
\begin{proof}
Using Lemma \ref{L3.1}, the proof is provided similarly to that of Lemma \ref{Lemma6}.
\end{proof}

\begin{thm}\label{T3.1} Let $1< p<\infty$, $ 0< \theta, \theta_0,\theta_1$, and 
 $0<\tau ,\tau_0, \tau_1\leqslant \infty$.
Then
 \begin{equation} \label{3.5}
\|f * g\|_{GL_{p,\tau_1}^{-\theta_1}(\Omega)} \lesssim  \|f\|_{GL_{p,\tau_0}^{\theta_0}(\Omega)}\|g\|_{GL^{-\theta}_{1,\tau}(\Omega)},
\end{equation}
if   $1+\frac1{\tau_1}+\theta_1=\frac1{\tau_0}-\theta_0+\frac1{\tau}+\theta$,    and
\begin{equation} \label{3.6}
\|f * g\|_{GL_{\infty,\tau_1}^{\theta_1}(\Omega)} \lesssim  \|f\|_{GL_{p,\tau_0}^{\theta_0}(\Omega)}\|g\|_{GL^{\theta}_{p',\tau}(\Omega)},
\end{equation}
if $1+\frac1{\tau_1}-\theta_1=\frac1{\tau_0}-\theta_0+\frac1{\tau}-\theta$. 
\end{thm}
\begin{proof}  Let $0<\varepsilon<1/p$,  $\bar\theta= 1+\frac1{\tau_1}-\frac1{\tau_0}-\frac1{\tau} $, $\frac{1}{q(-\varepsilon)}= \frac{1}{q}-\varepsilon$, and $1(-\varepsilon)= 1-\varepsilon$.  From the  inequality \eqref{3.4} and the relation $(1(-\varepsilon))'=1/\varepsilon$, we get 
$$
\varepsilon^{-\theta_1}\|f * g\|_{L_{p(-\varepsilon),\tau_1}(\Omega)} \lesssim \varepsilon^{-\bar\theta-\theta_1} \|f\|_{GL_{p,\tau_0}^{\theta_0}(\Omega)}\|g\|_{GL^{-\theta_0}_{1(-\varepsilon),\tau}(\Omega)}.
$$

Thus, we have 
\begin{align*}
\|f * g\|_{GL_{p,\tau_1}^{-\theta_1}(\Omega)} &\lesssim \varepsilon^{-\bar\theta-\theta_1} \inf_{0<\delta<1}\delta^{-\theta_0}\|g\|_{L_{1(-\varepsilon-\delta),\tau}(\Omega)} \|f\|_{GL_{p,\tau_0}^{\theta_0}(\Omega)} \\
 &\lesssim \varepsilon^{-\bar\theta-\theta_1-\theta_0} \|g\|_{L_{1(-2\varepsilon),\tau}(\Omega)} \|f\|_{GL_{p,\tau_0}^{\theta_0}(\Omega)},
\end{align*}
%where $\frac1{1(-\varepsilon)}=1-\varepsilon$.
which yields \eqref{3.5}, due to the arbitrariness of $\varepsilon>0$.

Now, it remains to show \eqref{3.6}. From the %duality theorem (see 
norm representation \eqref{10}, to be established in Theorem \ref{theo6.1},  it follows that
\begin{align*}
\|f*g\|_{G L_{\infty,\tau_1}^{\theta_1} (\Omega)}
&=\sup _{\|h\|_{G L_{1,\tau_1'}^{-\theta_1} (\Omega)}=1} \int_{\Omega} %[0,1]^n} 
(f* g)(y)h(y) dy\\
&=\sup _{\|h\|_{G L_{1,\tau_1'}^{-\theta_1} (\Omega) }=1} \int_{ \Omega} % [0,1]^n} 
f(x)(g*h)(x)dx.
\end{align*}
Applying the H\"older inequality  \eqref{Hold} we have 
$$
\|f*g\|_{G L_{\infty,\tau_1}^{\theta_1} (\Omega)}\leqslant \|f\|_{G L_{p,\tau_0}^{\theta_0} (\Omega)}
\sup _{\|h\|_{G L_{1,\tau_1'}^{-\theta_1} (\Omega)}=1} \|g*h\|_{G L_{p',\tau_0'}^{-\theta_0} (\Omega)}.
$$

Furthermore, using the inequality \eqref{3.5} and $\theta=\frac1{\tau_0}-\frac1{\tau_1}+\frac1{\tau}-1+\theta_1-\theta_0$, we arrive at
\begin{align*}
\|f*g\|_{G L_{\infty,\tau_1}^{\theta_1} (\Omega)} & \lesssim \|f\|_{G L_{p,\tau_0}^{\theta_0} (\Omega)}
\sup _{\|h\|_{G L_{1,\tau_1'}^{-\theta_1} (\Omega)}=1} \|g\|_{G L_{p',\tau}^{\theta} (\Omega)}\|h\|_{G L_{1,\tau_1'}^{-\theta_1} (\Omega)} %\\
%&=\|f\|_{G L_{p,\tau_0}^{\theta_0} (\Omega)} \|g\|_{G L_{p',\tau}^{\theta} (\Omega)},
\end{align*}
which ends the proof.
\end{proof}

		\section{Interpolation theorems} \label{sec:5}

\begin{thm}\label{thm3} Let $\theta \in \mathbb{R_+}$ and $0<p_0<p_1<\infty$. Then
$$\left(G L_{p_0 ,q_0}^\theta(\Omega), G L_{p_1, q_1}^{\theta}(\Omega)\right)_{\eta, q} \hookrightarrow G L_{p,q}^\theta(\Omega),$$
where
$\frac{1}{p}=\frac{1-\eta}{p_0}+\frac{\eta}{p_1}, \quad \eta \in(0,1).$
\end{thm}

\begin{proof}%[Proof of Theorem \ref{thm3}]
Set $|\Omega|=1$ without loss of generality.
Let $f \in\left(G L_{p_0, q_0}^\theta, G L_{p_1, q_1}^\theta\right)_{\eta, q}$ and $f=f_0+f_1$ be an
arbitrary representation, where $f_0 \in G L_{p_0, q_0}^\theta(\Omega)$ and $f_1 \in G L_{p_1, q_1}^\theta(\Omega).$

Now we estimate $f^*(t)$. Let $\frac{1}{p_i}>\varepsilon>0, \; i=0,1.$
Thus, we have
$$
\begin{aligned}
 f^*(t) &\leqslant f_0^*\left(\frac{t}{2}\right)+f_1^*\left(\frac{t}{2}\right) \\
&=\left(\frac{t}{2}\right)^{-\left(\frac{1}{p_0}+\varepsilon  \right)}\left(\left\|f_0\right\|_{L_{p_0(\varepsilon), \infty}}+\left(\frac{t}{2}\right)^{-\left(\frac{1}{p_1}+\varepsilon  \right)}\|f\|_{L_{p_1(\varepsilon), \infty}}\right)\\
&  \leqslant  t^{-\frac{1}{p_1(\varepsilon)}}\left\|f_0\right\|_{L_{p_0(\varepsilon),q_0}}
+t^{-\frac{1}{p(\varepsilon)}}\|f\|_{L_{p_1(\varepsilon),q_1}}, \end{aligned}
$$
where  $\frac{1}{p_i(\varepsilon)}=\frac{1}{p_i}+\varepsilon , \quad i=0,1$.

Given the arbitrariness of the representation $f=f_0+f_1$ we have
\begin{align}
 f^*(t) &\lesssim t^{-\frac{1}{p_0(\varepsilon)}}  \inf_{f=f_0+f_1} (
\|f_{0}\|_{L_{p_{0}(\varepsilon), q_{0}}} + t^{\frac{1}{p_0(\varepsilon)}-\frac{1}{p_1(\varepsilon)}}\|f_{1}\|_{L_{p_{0}(\varepsilon), q_{0}}} )\nonumber \\
& =t^{-\frac{1}{p(\varepsilon)}} K\left(t^{\frac{1}{p_{0}(\varepsilon)}-\frac{1}{p_1(\varepsilon)}}, f\right).
\label{star1}
\end{align}

Then when $\frac{1}{p_1(\varepsilon)}=\frac{1}{p_1}+\varepsilon$, we obtain
\begin{align*}
& \|f\|_{G L_{p,q}^\theta}=\sup _{0<\varepsilon} \varepsilon^\theta\|f\|_{L_{p(\varepsilon),q}(\Omega)}=\sup _{0<\varepsilon} \varepsilon^\theta\left(\int_0^1\left(t^{\frac{1}{p(\varepsilon)}} f^*(t)\right)^q \frac{d t}{t}\right)^{1 / q} \\
& \leqslant \sup _{0<\varepsilon} \varepsilon^\theta\left(\int_0^1\left(t^{\frac{1}{p(\varepsilon)}-\frac{1}{p_0(\varepsilon)}} K\left(t^{\frac{1}{p_0(\varepsilon)}-\frac{1}{p_1(\varepsilon)}}, f\right)\right)^q \frac{d t}{t}\right)^{1 / q} \\
& =\sup _{0<\varepsilon} \varepsilon^\theta\left(\int_0^1\left(t^{\frac{1}{p}-\frac{1}{p_0}} K\left(t^{\frac{1}{p_0}-\frac{1}{p_1}}, f\right)\right)^q \frac{d t}{t}\right)^{\frac{1}{q}}= \\
& =\sup _{0<\varepsilon}\left(\varepsilon^\theta \int_0^1(t^{-\eta }- K(t, f))^q \frac{d t}{t}\right)^{\frac{1}{q}} \cdot\left(\frac{1}{p_0}-\frac{1}{p_1}\right)^{-\frac{1}{q}}\\
& \leqslant \left(\frac{1}{p_0}-\frac{1}{p_1}\right)^{-\frac{1}{q}}\cdot
\left( \int_0^1\left(t^{-\eta }-  \sup _{0<\varepsilon} \varepsilon^\theta  \inf_{f=f_0+f_1}
(\|f_{0}\|_{L_{p_{0}(\varepsilon), q_{0}}} + t\|f_{1}\|_{L_{p_{0}(\varepsilon), q_{0}}}) \right)^q \frac{d t}{t}\right)^{\frac{1}{q}}
\\
& \leqslant \left(\frac{1}{p_0}-\frac{1}{p_1}\right)^{-\frac{1}{q}}\cdot
\left( \int_0^1\left( t^{-\eta }-   \inf_{f=f_0+f_1} (
\|f_{0}\|_{L_{p_{0}, q_{0}}} + t\|f_{1}\|_{L_{p_{0}, q_{0}}})\right) ^q \frac{d t}{t}\right)^{\frac{1}{q}}
\\ & \lesssim \|f\|_{\left(G L_{p_0, q_0}^\theta(\Omega), G L_{p_1, q_1}^{\theta}(\Omega)\right)_{\eta, q}}. \qedhere
\end{align*}
\end{proof}
\begin{cor} Let $A_0$ and $A_1$ be a compatible pair and $1<p_0<p_1<\infty$. Let $T$ be a quasilinear operator such that
$$
T: A_0 \longrightarrow G L_{p_0, \infty}^\theta \text { with norm } M_0,
$$
$$T: A_1 \longrightarrow G L_{p_1, \infty}^\theta \text { with norm } M_1.$$
Then
 $T: A_{\theta, q} \longrightarrow G L_{p,q}^\theta$ with norm $\|T\| \leqslant c M_0^{1-\theta} M_1^\theta$.

\end{cor}

\begin{proof}
According to the real interpolation method, we have $$T: A_{\eta q} \longrightarrow\left(G L_{p_0 q_0}^\theta, G L_{p_1 q_1}^\theta\right)_{\eta q}.$$

Applying Theorem \ref{thm3} we arrive at
$$T: A_{\eta, q} \longrightarrow\left(G L_{p_0, q_0}^\theta, G L_{p_1, q_1}^\theta\right)_{\eta q} \hookrightarrow G L_{p,q}^\theta,$$
which ends the proof.
\end{proof}

\begin{cor} Let $0<p_0<p_1<\infty, \quad q_0 \neq q_1$. Let $T$ be a quasilinear operator such that
$$T: L_{p_0} \longrightarrow G L_{q_0, \infty}^\theta\text { with norm } M_0$$
$$T: L_{p_1} \longrightarrow G L_{q_1, \infty}^\theta \text { with norm }M_1.$$ Then
$T: L_{p,\tau} \longrightarrow G L_{q, \tau}^\theta$ and $$
\|T f\|_{G L_{q, \tau}^\theta} \leqslant c M_0^{1-\theta} M_1^\theta\|f\|_{ L_{p, \tau}} .
$$
\end{cor}

\begin{thm}\label{Theorem 4} Let $0<p_0<p_1<\infty, 0<q_0<q_1<\infty, 0 \leqslant \theta<\infty, 0<\tau\leqslant \infty$, and $|\Omega|<\infty$.
Let $T$ be a quasilinear operator such that
$$T: L_{p_i,t}(U) \longrightarrow L_{q_1, \infty}(\Omega)
 \; \text{with norm} \; M_i, \; i=1,2.$$
Then $T: G L_{p, \tau}^\theta \rightarrow G L_{q',\tau}^\theta \quad$ and $$\|T\| \leqslant c M_0^{1-0} M_1^\theta,$$ where
\begin{equation}\label{equation8}
\frac{1}{p}=\frac{1-\eta}{p_0}+\frac{\eta}{p_1}, \quad \frac{1}{q}=\frac{1-\eta}{q_0}+\frac{\eta}{q_1}, \,\eta \in(0,1),\, \tau \in(0,+\infty] .
\end{equation}
\end{thm}

\begin{proof}%[Proof of Theorem \ref{Theorem 4}]
According to the Marcinkiewicz--Calder\'on theorem, we have:
\begin{equation}\label{equation9}
\|T f\|_{L_\eta(\Omega)}\leqslant c M_0^{1-0} M_1^0\|f\|_{L_{\mathrm{p} ,\tau}} ,
\end{equation}
where $p, q, \eta, \tau$ satisfy \eqref{equation8}. Let $\varepsilon<\delta$ then there will be $\xi: \varepsilon=\xi\left(\frac{1}{q_0}-\frac{1}{q_1}\right)$
Let $\frac{1}{q(\varepsilon)}=\frac{1}{q}+\varepsilon=\frac{1-\theta}{q_0}+\frac{\theta}{q_1}+\xi\left(\frac{1}{q_0}-\frac{1}{q_1}\right)$ $=\frac{1-(\theta-\xi)}{q}+\frac{(\theta-\xi)}{q}.$
From inequality $\varepsilon<\delta$ follows that $\theta-\xi \in[0,1]$.
Then from inequalities  \eqref{equation9} and
$$
\|T f\|_{L_{q(\varepsilon) ,\tau}(\Omega)} \leqslant c M_0^{1-\theta} M_1^\theta\|f\|_{L_{\tilde{p}(\varepsilon),\tau}(U)},
$$
where
$$
\frac{1}{\tilde{p}}=\frac{1-(\theta+\xi)}{p_0}+\frac{(\theta-\xi)}{p_1}=\frac{1}{p}+\xi\left(\frac{1}{p_0}-\frac{1}{p_1}\right)=\frac{1}{p}+\varepsilon \gamma=\frac{1}{p(\varepsilon \gamma)}
$$
with
$$
\gamma=\frac{\frac{1}{p_0}-\frac{1}{p_1}}{\frac{1}{q_0}-\frac{1}{q_1}}.
$$
Thus, we have
$$
\|T f\|_{L_{q(\varepsilon), \tau}(\Omega)} \leqslant c M_0 \mu\|f\|_{L_{p(\varepsilon \gamma), \tau}(U)}.
$$

From the property 5, it follows that
\begin{align*}
 \|T f\|_{G L_{q,\tau}^\theta} &\asymp \sup_{0<\varepsilon<\delta} \varepsilon^\theta\|T f\|_{L_{q(\varepsilon),\tau}(U)}  \\
& \leqslant c M_0^{1-\theta} M_1^{\theta} \sup_{0<\varepsilon<\delta} \varepsilon^{\theta}\|f\|_{L_{ p(\varepsilon \gamma),\tau}(U)} \\
& =c \gamma^{-\theta} M_0^{1-\theta} M_1^\theta \sup _{0<\varepsilon<\delta \gamma} \varepsilon^\theta\|f\|_{L_{p(\varepsilon),\tau}(U)} \\
&  \asymp M_0^{1-\theta} M_1^\theta\|f\| _{G L_{p ,\tau}^\theta }.\qedhere
\end{align*}
\end{proof}

\section{K\"othe Dual spaces} \label{sec:6}

% \todo[inline]{\texttt{I think the previous version has an issue in this section and the next. From the theory of Banach Function spaces, as appearing in Bennet-Sharpley book, the associate space coincides with the dual space when the norm is absolutely continuous. It is known that there are functions in grand Lebesgue spaces for which the norm is not absolutely continuous; e.g., taking $L^{p)} ([0,1])$ and $f(x)= x^{-1/p}$, note that, for $E_n=(0,1/n)$, we have 
% \begin{align*}
% \|f \chi_{E_n}\|_{L^{p)}(0,1)} &= \sup_{0<\varepsilon < p-1} \left( \varepsilon \int_0^{1/n} x^{-(p-\varepsilon)/p} dx \right)^{\frac{1}{p-\varepsilon}} \\ 
% & = \sup_{0<\varepsilon < p-1} (p (1/n)^{\frac{\varepsilon}{p}} )^{\frac{1}{p-\varepsilon}} \\
% & \geq \lim_{\varepsilon \to 0} (p (1/n)^{\frac{\varepsilon}{p}} )^{\frac{1}{p-\varepsilon}}  = p ^{1/p} \nrightarrow 0.
%  \end{align*}}
%  \texttt{We should change \emph{dual space} to either \emph{associate space} or \emph{K\"othe dual space} and I think it will be enough.
%  }}

In this section, we study the K\"othe dual space, also known as the associate space, for grand Lorentz spaces and a variant of grand Lebesgue spaces.

\subsection{Grand Lorentz spaces} It turns out that the K\"othe dual space for the grand Lorentz space has a simple characterisation given by $G L_{p,q}^\theta (\Omega)=(G L_{p^{\prime}, q^{\prime}}^{-\theta} (\Omega))^{\prime}.$ 

\begin{thm}\label{theo6.1}
Let $1\leqslant p, q<\infty$ and  $\theta>0$. Then
\begin{equation}\label{10}
\|f\|_{G L_{p,q}^\theta (\Omega)} \asymp \sup _{\|g\|_{G L_{p',q'}^{-\theta} (\Omega) }=1} \int_{\Omega} f(x) g(x) dx,
\end{equation}
that is,
\begin{equation}\label{11}
G L_{p,q}^\theta (\Omega)=\left(G L_{p^{\prime}, q^{\prime}}^{-\theta} (\Omega)\right)^{\prime}.
\end{equation}
% \todo[inline]{\texttt{Maybe better to take the next two lines, and the corresponding proof, since the content already appears in \eqref{1.4}.}

% Moreover, for $\Omega$ such that $|\Omega|=1$, we have 
% \begin{equation}\label{12}
% \|f\|_{G L_{p,q}^{-\theta} (\Omega)}\geqslant \sup _{\|g\|_{G L_{p',q'}^{\theta} (\Omega)}=1} \int_{\Omega} f(x) g(x) dx  \gtrsim \left(\int_0^1\left(t^{\frac1{p}}|\ln t|^{\theta} f^*(t)\right)^{q} \frac{d t}{t}\right)^{\frac1{q}}.
% \end{equation}}
\end{thm}

\begin{proof}
Taking $0< \varepsilon<1-\frac{1}{p'}$, $ \frac1{p(\varepsilon)}:=\frac1p+\varepsilon$, $ \frac1{p'(-\varepsilon)}:=\frac{1}{p'}-\varepsilon$, and noticing that $p'(\varepsilon)= (p(\varepsilon))'$, we have
\begin{align}\label{eq:dual}
\int_{\Omega} f(x) g(x) d x &\leq
\|f\|_{L_{p(\varepsilon),q} (\Omega)}\|g\|_{L_{p^{\prime}(-\varepsilon), q^{\prime}} (\Omega)} \nonumber\\
&=\varepsilon^\theta\|f\|_{L_{p(\varepsilon), q}(\Omega)}\varepsilon^{-\theta}\|g\|_{L_{p^{\prime}(-\varepsilon), q^{\prime}} (\Omega) } \\ 
%&\leqslant \left(\sup _{0 \leqslant \varepsilon \leqslant 1} \varepsilon^\theta\|f\|_{L_{p(\varepsilon), q}}\right) \cdot \varepsilon^{-\theta}\|g\|_{L_{p^{\prime}(-\varepsilon), q^{\prime}}}\\
&=\|f\|_{G L_{p,q}^\theta(\Omega)} \cdot \varepsilon^{-\theta}\|g\|_ {L_{p^{\prime}(-\varepsilon), q^{\prime}}(\Omega)}. \nonumber
\end{align}
By the arbitrariness of $\varepsilon$ in \eqref{eq:dual}, we conclude that  
%Thus, for any $0<\varepsilon<1-\frac{1}{p'}$, we have
%$$
%\int_{\Omega} f(x) g(x) d x \leqslant\|f\|_{G L_{p,q}^\theta(\Omega)} \varepsilon^{-\theta}\|g\|_{L_{p^{\prime}(-\varepsilon), q}(\Omega)}.
%$$
%Therefore, we get
$$
\int_{\Omega} f(x) g(x) d x \leqslant\|f\|_{G L_{p,q}^\theta (\Omega)}\|g\|_{G L_{p^{\prime}, q^{\prime}}^{-\theta}(\Omega)}, 
$$
hence
$$
\sup _{\|g\|_{G L_{p',q'}^{-\theta} (\Omega)}=1} \int_{\Omega} f(x) g(x) d x \leqslant\|f\|_{G L_{p,q}^\theta(\Omega)}
$$
and
$$
\sup _{\|g\|_{G L_{p',q'}^{\theta}(\Omega)}=1} \int_{\Omega} f(x) g(x) d x \leqslant\|f\|_{G L_{p,q}^{-\theta}(\Omega)} .
$$

To prove the converse inequality, note that 
\begin{align*}
\sup _{\|g\|_{G L_{p^{\prime} q^{\prime}}^{-\theta}(\Omega)}=1} \int_{\Omega} f(x) g(x) d x
&=\sup _{g \neq 0} \frac{\int_\Omega f(x) g(x) d x}{\|g\|_{G L_{p^{\prime}, q^{\prime}}^{-\theta}(\Omega)}}\\
&=\sup _{g \neq 0} \frac{\int_\Omega f(x) g(x) d x}{\inf _{0<\varepsilon<\frac{1}{p}} \varepsilon^{-\theta}\|g\|_{L_{p^{\prime}(-\varepsilon), q^{\prime}}(\Omega)}}\\
&=\sup _{g \neq 0}\sup _{0<\varepsilon<\frac{1}{p}} \varepsilon^{\theta} \frac{\int_\Omega f(x) g(x) d x}{\|g\|_{L_{p^{\prime}(-\varepsilon), q^{\prime}}(\Omega)}}\\
&\asymp \sup _{0<\varepsilon<\frac{1}{p}} \varepsilon^{\theta} \|f\|_{L_{p(\varepsilon), q}(\Omega)}\\ & \asymp \|f\|_{G L_{p,q}^\theta(\Omega)},
\end{align*}
which proves \eqref{10}.
\end{proof}
% \todo[inline]{\texttt{I suggest to delete the remaining texts, since the same result already appeared:}

% Assume that $|\Omega|=1$. Using Lemma \ref{Lemma 1}, we obtain
% \begin{align*}
% \sup _{\|g\|_{G L_{p^{\prime} q^{\prime}}^{\theta}(\Omega)}=1} \int_{\Omega} f(x) g(x) d x
% &=\sup _{g \neq 0} \frac{\int_\Omega f(x) g(x) d x}{\|g\|_{G L_{p^{\prime}, q^{\prime}}^{\theta}(\Omega)}}\\
% &\gtrsim \sup _{g \neq 0} \frac{\int_\Omega f(x) g(x) d x}{\left(\int_0^1\left(t^{\frac{1}{p'}}|\ln t|^{-\theta} g^*(t )\right)^{q'} \frac{d t}{t}\right)^{\frac{1}{q'}}}.
% \end{align*}
% Since  the integral $\int_0^1\left(t^{\frac{1}{p}}|\ln t|^{-\theta} g^*(t)\right)^q \frac{d t}{t}$ depends on the nonincreasing rearrangement of $g$ we have
% \begin{align*}
%     \sup _{\|g\|_{G L_{p^{\prime} q^{\prime}}^{\theta}(\Omega)}=1} \int_{\Omega} f(x) g(x) d x
% &\gtrsim \sup _{g \neq 0} \frac{\int_0^1 f^*(t) g^*(t) d t}{\left(\int_0^1\left(t^{\frac{1}{p^{\prime}}}|\ln t|^{-\theta} g^*(t)\right)^{q^{\prime}} \frac{d t}{t}\right)^{\frac1{q^{\prime}}}}\\
% &=\left(\int_0^1\left(t^{\frac1{p}}|\ln t|^{\theta} f^*(t)\right)^{q} \frac{d t}{t}\right)^{\frac1{q}}. 
% \end{align*}
% }

\subsection{Grand Lebesgue spaces} \label{sec:7}
We now introduce grand Lebesgue spaces $GL^\theta_p (\Omega)$ using same approach as in Definition \ref{def:1}. They coincide, as shown in Lemma \ref{duallem1},  with the grand Lebesgue spaces $L^{p),\theta}(\Omega)$, % of Iwaniec and Sbordone, 
as 
defined in  \eqref{eq:IwaniecSbordone}.

 For $\theta \in \mathbb{R}$ and $0<p<\infty$, %$|\Omega|<\infty$,  a $f$ a measurable function, 
 we define 
 $G L_p^\theta(\Omega)$ as the set of $f \in L^0 (\Omega)$ having finite norm:
$$
\|f\|_{G L_p^\theta(\Omega)}=\left\{\begin{array}{l}
\displaystyle \sup _{0<\varepsilon<1} \varepsilon^{\theta}\|f\|_{L_{p(\varepsilon)}}, \; \theta \geqslant 0,\\
\displaystyle \inf_{0<\varepsilon<\frac{1}{p}} \varepsilon^\theta\|f\|_{L_{p(-\varepsilon)}}, \; \theta<0,
\end{array}\right.
$$
where
$$
\frac{1}{p(\varepsilon)}=\frac{1}{p}+\varepsilon \quad  \text { and } \quad  \frac{1}{p(-\varepsilon)}=\frac{1}{p}-\varepsilon.
$$
\begin{thm}\label{dualspacethm} Let $0<p<\infty$ and $\theta > 0$. Then
$$
\left(G L_{p^{\prime}}^{-\theta} 
(\Omega) \right)^{\prime}=G L_p^\theta(\Omega) 
$$
and
$$
\sup _{\|g\|_{G L^{-\theta}
_{p'} (\Omega)}=1} \int_{\Omega} f(x) g(x) d x \asymp\|f\|_{G L_p^\theta (\Omega)}.
$$
\end{thm}

\begin{proof}%[Proof of Theorem \ref{dualspacethm}]
The proof follows in a similar manner to that of Theorem \ref{theo6.1}. % for $G L_{p,q}^{\theta}(\Omega)$.
\end{proof}

\begin{lem} \label{duallem1} Let $0<\theta$ and $0<p<\infty$. %,|\Omega|<\infty$. 
Then
\begin{equation}\label{eq:grandandgrand}
G L_p^\theta(\Omega)=L^{p), \theta}(\Omega).
\end{equation}
% and
% \begin{equation}\label{eq:6.6}
%     G L_{p}^{-\theta}(\Omega)=L^{(p,  \theta}(\Omega).
% \end{equation}
% \todo[inline]{I suggest taking out \eqref{eq:6.6}, because small spaces have a quite complicated norm and we do not give any proof. Maybe this can be given in another paper. }
\end{lem}
\begin{proof}
Let $0<\varepsilon<p / 2$. Then $\frac{1}{p-\varepsilon}-\frac{1}{p}=\frac{\varepsilon}{(p-\varepsilon) p}$ and therefore
$$
\frac{\varepsilon}{p^2}<\frac{\varepsilon}{(p-\varepsilon) p} < \frac{2 \varepsilon}{p^2}.
$$
%Recall that, for $|\Omega|<\infty$ and $q<p$ we have\(L_p(\Omega) \hookrightarrow L_q(\Omega). \)
Taking $\frac{1}{p(\eta)}:=\frac{1}{p}+\eta$, we get
\begin{equation}\label{eq*}
L_{p\left(\frac{\varepsilon}{p^2}\right)}(\Omega)
\hookrightarrow
L _ { p - \varepsilon } ( \Omega )\hookrightarrow L_{p\left(\frac{2\varepsilon}{p^2}\right)}(\Omega),
\end{equation}
 which follows from the embedding \(L_p(\Omega) \hookrightarrow L_q(\Omega) \), where $q<p$ and $|\Omega|<\infty$.  

 $$
\begin{aligned}
 \|f\|_{G L_p^\theta(\Omega)}&=\sup _{0<\varepsilon<\frac{1}{p}} \varepsilon^\theta\|f\|_{L_{p(\varepsilon)}(\Omega)} \\
& =\sup _{0<\varepsilon<\frac{p}{2}}\left(\frac{2 \varepsilon}{p^2}\right)^\theta\|f\|_{L_p\left(\frac{2 \varepsilon}{p^2}\right)(\Omega)}  \\
& \lesssim \sup _{0<\varepsilon<\frac{p}{2}}\left(\frac{2 \varepsilon}{p^2}\right)^\theta\|f\|_{L_{p-\varepsilon}(\Omega)} \\
& =\left(\frac{2}{p^2}\right)^\theta\|f\|_{L^{p), \theta}(\Omega)}.
\end{aligned}
$$

Conversely, by using \eqref{eq*}, we obtain
 \begin{align*}
 \|f\|_{G L_p^\theta (\Omega)} &=\sup _{0<\varepsilon<\frac{1}{p}} \varepsilon^\theta\|f\|_{L_{p(\varepsilon)(\Omega)}} \\
& =\sup _{0<\varepsilon<p}\left(\frac{\varepsilon}{p^2}\right)^\theta\|f\|_{L_{p\left(\frac{\varepsilon}{p^2}\right)}(\Omega) } \\
& \gtrsim \sup _{0<\varepsilon<p}\left(\frac{\varepsilon}{p^2}\right)^\theta\|f\|_{L_{p-\varepsilon}(\Omega)} \\
& =\left(\frac{1}{p^2}\right)^\theta \sup _{0<\varepsilon<p} \varepsilon^\theta\|f\|_{L_{p-\varepsilon}(\Omega)} \\
&=\frac{1}{p^{2 \theta}}\|f\|_{L^{p), \theta}(\Omega)}. \qedhere 
\end{align*}
\end{proof}

From Theorem \ref{dualspacethm} and Lemma \ref{duallem1}, we obtain the following.

\begin{thm}\label{dualthm2} Let $\theta > 0$ and $1<p<\infty$. Then
$$
L^{p), \theta}(\Omega)=\left(L^{p'(\theta}(\Omega)\right)^{\prime}.
$$
\end{thm}

% \begin{proof}[Proof of Theorem \ref{dualthm2}]

% The proof follows from Theorem \ref{dualspacethm} and Lemma \ref{duallem1}.

% \end{proof}

\textbf{Author Contributions} Authors have been discussing and working together on the manuscript, contributing
equally to the content, presentation, and review of the manuscript.

\textbf{Funding} This research is funded by the Science Committee of MSHE of Kazakhstan (Grant No. BR21882172). This research was also funded by Nazarbayev University under CRP grant
20122022CRP1601. The research of H.R. is supported by the UPAR grant G00004572, United Arab Emirates University.

%the European Union’s Horizon 2020 research and innovationprogramme under the Marie Skłodowska-Curie grant agreement no 101034255.

\textbf{Data availability} Data sharing is not applicable to this article, as no datasets were generated or analysed
during the current study.

\textbf{Conflict of interest} The authors have no Conflict of interest as defined by Springer, or other interests that
might be perceived to influence the results and/or discussion reported in this paper.

\textbf{Ethical approval} Not applicable.

\end{document}